\documentclass[a4paper,11pt]{amsart}
\usepackage{amsmath}
\usepackage{amssymb}
\usepackage{graphicx}
\usepackage{amscd}
\usepackage[all]{xy}
\usepackage{amsfonts}
\usepackage{mathrsfs}
\usepackage{ifthen}
\usepackage{amsthm}
\usepackage{leqno}

\newcommand\Z{{\mathbb Z}}
\newcommand\C{{\mathbb C}}

\newcommand\Ps{{\mathbb P}}

\newcommand\Ec{{\mathcal E}}
\newcommand\F{{\mathcal F}}

\newcommand\Ic{{\mathcal I}}

\newcommand\Lca{{\mathcal L}}

\newcommand\Nc{{\mathcal N}}
\newcommand\Oc{{\mathcal O}}
\newcommand\Pc{{\mathcal P}}
\newcommand\Sc{{\mathcal S}}

\newcommand\codi{\operatorname{codim}}

\newcommand\Supp{\operatorname{Supp}}

\newcommand\iso{\kern.35em{\raise3pt\hbox
{$\sim$}\kern-1.1em\to}\kern.3em}

\newcommand\WIT{\operatorname{WIT}}
\newcommand\IT{\operatorname{IT}}
\theoremstyle{plain}
\newtheorem{thm}{Theorem}[section]

\newtheorem{corol}[thm]{Corollary}
\newtheorem{lemma}[thm]{Lemma}
\newtheorem{prop}[thm]{Proposition}
\theoremstyle{definition}
\newtheorem{defin}[thm]{Definition}

\theoremstyle{remark}
\newtheorem{remark}[thm]{Remark}
\numberwithin{equation}{section}

\begin{document}

\title[Trisecant Identity through the Fourier-Mukai transform]{A proof of the Trisecant Identity through the Fourier-Mukai transform.}

\author[D. Hern\'andez Serrano, J. M. Mu\~noz Porras and F. J. Plaza Mart\'{\i}n]{D. Hern\'andez Serrano \\ J. M. Mu\~noz Porras \\  F. J. Plaza Mart\'{\i}n}

\address{Departamento de Matem\'aticas, Universidad de
Salamanca,  Plaza
        de la Merced 1-4
        \\
        37008 Salamanca. Spain.
        \\
         Tel: +34 923294460. Fax: +34 923294583}
\address{IUFFYM. Instituto Universitario de F\'{\i}sica Fundamental y Matem\'aticas, Universidad de Salamanca, Plaza de la Merced s/n\\ 37008 Salamanca. Spain.}
\date\today
\thanks{
\indent This work is partially supported by research contracts
MTM2006-0768 of DGI and  SA112A07 of JCyL. The first
author is also supported by MTM2006-04779. \\
\indent {\it E-mail addresses}: dani@usal.es, jmp@usal.es,fplaza@usal.es
}

\email{dani@usal.es}
\email{jmp@usal.es}
\email{fplaza@usal.es}

\begin{abstract}
Using the technique of the Fourier-Mukai transform we give an explicit set of generators of the ideal defining an algebraic curve as a subscheme of its Jacobian. Essentially, these ideals are generated by the Fay's trisecant identities.
\end{abstract}

\maketitle


\section{Introduction.}

Let $C$ be a smooth algebraic curve of genus $g$ defined over a field $k$ and $J$ its Jacobian. Our goal is to give explicit equations of the subschemes $W^{i}\subset J$ ($i\leq g$) defined by images of the Abel morphisms $S^{i}C \to J$. In particular, for the case $i=1$ we prove that the equations of $C\iso W^{1}$ as a subvariety of $J$ are generated by the trisecant identities proved by Fay \cite{Fay}. Some of the results presented in this paper were previously proved in \cite{KM} but we give a different proof and approach and we also aim to clarify some imprecisions in that paper. 

Finally, we give a reformulation of the trisecant conjecture characterizing Jacobians recently proved by Krichever \cite{K3}. We hope that our methods will be useful for understanding the proof of this conjecture in geometric terms.

Section $2$ offers a brief overview on abelian varieties and the Fourier-Mukai transform for such varieties. Section $3$ is devoted to giving the main result of the paper; using the theory of the Fourier-Mukai transform we explicitly compute global equations for certain subschemes of a principally polarized abelian variety. The importance of these equations becomes apparent in section $4$, where the case of Jacobians of smooth algebraic curves is studied. It turns out that the subschemes mentioned above are translations of certain symmetric products of the curve and that the equations computed generalize Fay's trisecant identity and Gunning's relations. Finally, in section $5$, we reformulate Krichever's result on the classical Riemann-Schottky problem in terms of the theory developed in the previous sections.

\clearpage

\section{Abelian varieties and the Fourier-Mukai transform.}

Everything in this section is extracted from \cite{MumAb,BBHR}.

\subsection{Abelian varieties.}\quad 

Let $X$ be an abelian variety of dimension $g$ over a field $k$ of characteristic $p\neq 2$. We shall use an additive notation and will denote:
\begin{align*}
&\begin{aligned}
m\colon X\times X &\to X \mbox{ (the group law)}\\
(x,y)&\mapsto x+y
\end{aligned}&&\\[.2cm]
&\begin{aligned}
\iota_{X} \colon X&\to X \mbox{ (the inverse morphism)}\\
x&\mapsto -x
\end{aligned}&&\\[.2cm]
&\begin{aligned}
\tau_{x} \colon X &\to X \mbox{ (the translation morphism)}\\
y&\mapsto \tau_{x}(y):=x+y
\end{aligned}&&
\end{align*}
The neutrum element of $X$ will be denoted $0$, and the projections onto the factors of $X\times X$ will be written with $\pi_{1}$ and $\pi_{2}$.

Let $\hat X$ be the \emph{dual abelian variety} of $X$, which represents the degree $0$ Picard functor of $X$. Therefore, there exists a universal line bundle $\Pc$ on $X\times \hat X$, called the \emph{Poincar\'e bundle}. Universality means that given a variety $T$ and a line bundle $\Lca$ on $X\times T$ (such that the first Chern class of its restriction to the fibers of $\pi_{T}\colon X\times T \to T$ vanishes), there exists a unique morphism $\varphi \colon T \to \hat X$ such that:
$$\Lca \iso (1\times \varphi)^{\ast} \Pc \otimes \pi_{T}^{\ast}\Nc\,,$$
where $\Nc$ is a line bundle on $T$. Thus, if a point $\xi \in \hat X$ corresponds to a line bundle $\Lca$ on $X$, one has:
$$\Pc_{\xi}=\Pc_{|X\times \{\xi\}}\iso \Lca\,.$$
Analogously, we shall denote $\Pc_{x}=\Pc_{|\{x\}\times \hat X}$ if $x\in X$.

Let us normalize the Poincar\'e bundle by the condition that $\Pc_{0}=\Pc_{|\{0\}\times \hat X}$ is the trivial line bundle on $\hat X$.

Let us consider the line bundle $m^{\ast}\Lca \otimes \pi_{1}^{\ast}\Lca^{-1}$ on $X\times X$. By the universal property of the Poincar\'e bundle and its normalization, there exists a unique morphism:
$$\varphi_{\Lca}\colon X \to \hat X$$
such that:
$$(1\times \varphi_{\Lca})^{\ast} \Pc \iso m^{\ast}\Lca \otimes \pi_{1}^{\ast}\Lca^{-1} \otimes \pi_{2}^{\ast}\Lca^{-1}\,,$$
and one checks that $\varphi_{\Lca}(x)=\tau_{x}^{\ast}\Lca \otimes \Lca^{-1}$.

\subsection{The Fourier-Mukai transform for abelian varieties.}\quad 

Let us denote by $D^{b}(X)$ (resp. $D^{b}(\hat X)$) the derived category of bounded coherent complexes on $X$ (resp. on $\hat X$). We shall denote the natural projections by $\pi_{X}\colon X\times \hat X \to X$ and $\pi_{\hat X}\colon X\times \hat X \to \hat X$.

\begin{thm}\cite[Theorem 2.2]{Muk1}
The integral functor defined by $\Pc$:
\begin{align*}
\Sc \colon D^{b}(X) &\to D^{b}(\hat X)\\
\Ec^{\bullet}& \mapsto \Sc(\Ec^{\bullet}):={\bf R} \pi_{\hat X \ast}(\pi_{X}^{\ast}(\Ec^{\bullet})\otimes \Pc)
\end{align*}
is a Fourier-Mukai transform; that is, an equivalence of categories. 
\end{thm}


\begin{defin}\cite[Def. 1.6]{BBHR}
A coherent sheaf $\F$ on $X$ is $\WIT_{i}$ if its Fourier-Mukai transform reduces to a single coherent sheaf $\hat \F$ located in degree $i$; that is $\Sc(\F)\iso \hat \F[-i]$. We shall say that $\F$ is $\IT_{i}$ if in addition $\hat \F$ is locally free.
\end{defin}

Let $\Lca$ be an ample line bundle on $X$. From \cite[Prop. 3.19]{BBHR} it follows that $\Lca$ is $\IT_{0}$ and thus:
$$\Sc(\Lca)=\pi_{\hat X \ast}(\pi_{X}^{\ast}\Lca \otimes \Pc)$$
and this is a locally free sheaf on $\hat X$. The following fact is also well-known, but since we shall make extensive use of this kind of computation along this paper we shall not omit its proof. 
\begin{lemma}\label{eq:l:1}
$\varphi_{\Lca}^{\ast}\Sc(\Lca)\iso \Gamma(X,\Lca)\otimes_{k}\Lca^{-1}$.
\end{lemma}
\begin{proof}
\begin{align*}
\varphi_{\Lca}^{\ast}\Sc(\Lca)&\iso \varphi_{\Lca}^{\ast}\pi_{\hat X \ast}(\pi_{X}^{\ast}\Lca \otimes \Pc)\iso \pi_{2 \ast}\big((1\times \varphi_{\Lca})^{\ast}(\pi_{X}^{\ast}\Lca \otimes \Pc)\big)\iso \\
&\iso \pi_{2 \ast}(\pi_{1}^{\ast}\Lca \otimes (1\times \varphi_{\Lca})^{\ast}\Pc)\iso \pi_{2 \ast}(m^{\ast}\Lca \otimes \pi_{2}^{\ast}\Lca^{-1})\iso \\
&\iso \pi_{2 \ast}\big((m,\pi_{2})^{\ast}(\pi_{1}^{\ast}\Lca \otimes \pi_{2}^{\ast}\Lca^{-1})\big)\iso \pi_{2 \ast}(\pi_{1}^{\ast}\Lca \otimes \pi_{2}^{\ast}\Lca^{-1})\iso \\
&\iso \pi_{2 \ast}\pi_{1}^{\ast}\Lca \otimes \Lca^{-1}\iso \Gamma(X,\Lca)\otimes_{k}\Lca^{-1}\,.
\end{align*}
\end{proof}

Let us now recollect some definitions and facts that will be needed later on. 

\begin{defin}\label{eq:d:*}
The \emph{Pontrjagin product} of two coherent sheaves $\Ec$ and $\F$ on $X$ is the sheaf:
$$\Ec  \star \F := m_{\ast}(\pi_{1}\Ec \otimes \pi_{2}\F)\,.$$
This product has a derived functor:
\begin{align*}
\overset{\bf R}{\star} \colon D^{b}(X)\times D^{b}(X) &\to D^{b}(X)\\
(\Ec^{\bullet},\F^{\bullet}) & \mapsto \Ec^{\bullet}\overset{\bf R}{\star} \F^{\bullet}:={\bf R}m_{\ast}(\pi_{1}^{\ast} \Ec^{\bullet} \overset{\bf L}{\otimes} \pi_{2}^{\ast}\F^{\bullet})
\end{align*}
\end{defin}
\begin{prop}\label{eq:p:*}\cite{Muk1}\cite[Prop.3.13]{BBHR}
The Fourier-Mukai transform intertwines the Pontrjagin and the tensor product; that is:
$$\Sc(\Ec^{\bullet}\overset{\bf R}{\star} \F^{\bullet})\iso \Sc(\Ec^{\bullet})\overset{\bf L}{\otimes} \Sc(\F^{\bullet})\quad \quad \Sc(\Ec^{\bullet} \overset{\bf L}{\otimes} \F^{\bullet})\iso \Sc(\Ec^{\bullet}) \overset{\bf R}{\star} \Sc(\F^{\bullet})[g]\,.$$
\end{prop}

\begin{defin}\label{eq:d:reg}
A sheaf $\Ec$ on $X$ will be called \emph{Mukai regular}, or simply \emph{M-regular}, if:
$$\codi \big(\Supp R^{i}\Sc(\Ec)\big)>i\, \quad \forall i>0\,.$$ 
Sheaves satisfying the $\IT_{0}$ condition are trivially M-regular.
\end{defin}

\begin{thm}\cite[Thm.2.4]{PP1}
Let $\Ec$ be a coherent sheaf and $\Lca$ an invertible sheaf supported on a subvariety $Y$ of the abelian variety $X$ (possibly $X$ itself). If both $\Ec$ and $\Lca$ are M-regular as sheaves on $X$, then $\Ec \otimes \Lca$ is generated by its global sections. 
\end{thm}
\begin{corol}\label{eq:c:PP1}\cite{PP1}
Let $(X,\Theta)$ be a polarized abelian variety and $\Ec$ a coherent sheaf on $(X,\Theta)$. If $\Ec(-\Theta)$ is M-regular, then $\Ec$ is generated by its global sections.
\end{corol}

\section{Global equations.}

\subsection{First results.}\quad

Let $H$ be a closed and finite subscheme of $X$ and let $\Lca$ be an ample line bundle on $X$. Consider the following exact sequence of sheaves on $X$:
$$0\to \Ic_{H}\otimes \Lca \to \Lca \to \Lca_{|H} \to 0\,,$$
where $\Ic_{H}$ denotes the sheaf of ideals of $H$. Since $\Lca_{|H}$ is concentrated over a $0$-dimensional variety, using Grauert's theorems and base change one can show that $\Sc^{j}(\Lca_{|H})=0$ for all $j>0$ and moreover:
$$\Sc^{0}(\Lca_{|H})=\pi_{\hat X \ast}\big((\pi_{X}^{\ast}\Lca \otimes \Pc)_{|H\times \hat X}\big)$$
is a locally free sheaf on $\hat X$; that is, $\Lca_{|H}$ is $\IT_{0}$. Therefore, the restriction map $\Lca \to \Lca_{H}$ induces a morphism of locally free sheaves on $\hat X$ between its transforms:
\begin{equation}\label{eq:e:alpha}
\alpha(\Lca,H) \colon \Sc(\Lca) \to \Sc(\Lca_{|H})
\end{equation}
\begin{defin}\label{eq:d:Z}
For every non-negative integer $i$, we shall write $Z^{i}(\Lca,H)$ to denote the closed subscheme of $\hat X$ defined by $\Lambda^{i} \alpha (\Lca,H)=0$.
\end{defin}
\begin{remark}
Notice that since $\Sc(\Lca)$ is not generetad by its global sections, the equations $\Lambda^{i} \alpha (\Lca,H)=0$ are not global. We will give a procedure to explicitly compute the global equations of $Z^{i}(\Lca,H)$.
\end{remark}

Let $\Delta\subset X\times X$ denote the diagonal and let $(H,0)+\Delta\subset X\times X$ be the image of $H\times X$ under the map $(m,\pi_{2})\colon X\times X\to X\times X$. Let $\beta(\Lca,H)$ be the following morphism of sheaves on $X$:
\begin{equation}\label{eq:e:beta}
\beta(\Lca,H)\colon \Gamma(X,\Lca)\otimes_{k}\Oc_{X} \to \pi_{2 \ast}(\pi_{1}^{\ast}\Lca_{|(H,0)+\Delta})\,,
\end{equation}
defined by restriction.
\begin{defin}\label{eq:d:U}
For every non-negative integer $i$, we shall write $U^{i}(\Lca,H)$ to denote the closed subscheme of $X$ defined by $\Lambda^{i} \beta (\Lca,H)=0$.

\end{defin}

\begin{lemma}\label{eq:l:2}One has the following diagram:
$$\xymatrix{
\varphi_{\Lca}^{\ast} \Sc(\Lca)\otimes_{\Oc_{X}}\Lca \ar[rrr]^{\varphi_{\Lca}^{\ast}\big(\alpha(\Lca,H)\big)\otimes 1} \ar[d]^{\wr}& & & \varphi_{\Lca}^{\ast} \Sc(\Lca_{|H})\otimes_{\Oc_{X}}\Lca \ar[d]^{\wr}\\
\Gamma(X,\Lca)\otimes_{k}\Oc_{X} \ar[rrr]^{\beta(\Lca,H)}& & &\pi_{2 \ast}(\pi_{1}^{\ast}\Lca_{|(H,0)+\Delta})
}$$
\end{lemma}
\begin{proof}
This follows from Lemma \ref{eq:l:1}.
\end{proof}

\begin{corol}\label{eq:c:3}
The preimage of $Z^{i}(\Lca,H)$ by the morphism $\varphi_{\Lca}$ is precisely $U^{i}(\Lca,H)$.
\end{corol}

Let us note that the above approach allows us to write down {\sl local} equations for the subschemes $Z^{i}(\Lca,H)$, whose geometric interpretation will be given in  Section \ref{eq:s:Jac} for the case of Jacobian varieties. On the other hand, {\sl global} equations for $U^{i}(\Lca,H)$ are already at our disposal, and these subschemes are related to the relative position of the points of $H$. Our next task, motivated by Lemma~\ref{eq:l:2}, consists of comparing these subschemes as well as obtaining {\sl global} equations for $Z^{i}(\Lca,H)$.

\subsection{Global equations.}\quad

Henceforth $X$ will denote a  principally polarized abelian variety $(X,\Theta)$, where $\Theta$ is a symmetric polarization and $\Lca$ will be a line bundle algebraically equivalent to $\Oc_{X}(2\Theta)$. Recall that the principal polarization $\Oc_{X}(\Theta)$ defines an isomorphism $\varphi_{\Oc_{X}(\Theta)}\colon X \iso \hat X$. From now on $D^{b}(X)$ and $D^{b}(\hat X)$ will be identified. Under this identification, the Poincar\'e bundle on $X\times X$, which we shall still denote by $\Pc$, takes the shape:
$$\Pc \overset{not}{=}(1\times \varphi_{\Oc_{X}(\Theta)})^{\ast} \Pc \iso m^{\ast}\Oc_{X}(\Theta) \otimes \pi_{1}^{\ast}\Oc_{X}(-\Theta) \otimes \pi_{2}^{\ast}\Oc_{X}(-\Theta)\,.$$
We shall use $Z^{i}(\Lca,H)$ to denote the subscheme of $X$ defined by $\varphi_{\Oc_{X}(\Theta)}^{-1}(Z^{i}(\Lca,H))$. Under the isomorphism $\varphi_{\Oc_{X}(\Theta)}\colon X \iso \hat X$, one has the following commutative diagram:
$$\xymatrix{
X \ar[r]^{\varphi_{\Lca}} \ar[d]_{2_{X}} & \hat X \\
X \ar[ur]_{\varphi_{\Oc_{X}(\Theta)}} &
}$$
where $2_{X}$ denotes the multiplication by $2$. We shall use the same notation both for $\varphi_{\Lca}$ and $2_{X}$.
\begin{thm}\label{eq:t:gg}
The sheaf $\Sc(\Lca)\otimes \Lca$ is generated by its global sections.
\end{thm}
\begin{proof}
For the sake of simplicity, one can reduce to the case in which $\Lca=\Oc_{X}(2\Theta)$. By Corollary \ref{eq:c:PP1}, it suffices to prove that: 
$$\Sc(\Lca)\otimes \Lca \otimes \Oc_{X}(-\Theta)\iso \Sc(\Lca)\otimes \Oc_{X}(\Theta)$$
is M-regular (see Definition \ref{eq:d:reg}). Indeed, let us set:
\begin{align*}
\mu \colon X\times X &\to X\\
(x,y)&\mapsto x-y 
\end{align*}
We have:
\begin{align*}
\Sc(\Lca)\otimes \Oc_{X}(\Theta)&=\pi_{2\ast}\big(\pi_{1}^{\ast}\Lca \otimes \Pc \big)\otimes \Oc_{X}(\Theta)\iso \\
&\iso \pi_{2\ast}\big(\pi_{1}^{\ast}\Oc_{X}(\Theta)\otimes m^{\ast}\Oc_{X}(\Theta)\big)\iso \\
&\iso \pi_{2\ast}\big((m,\pi_{1})^{\ast}(\pi_{2}^{\ast}\Oc_{X}(\Theta)\otimes \pi_{1}^{\ast}\Oc_{X}(\Theta)\big)\iso \\
&\iso \mu_{\ast}\big(\pi_{2}^{\ast}\Oc_{X}(\Theta)\otimes \pi_{1}^{\ast}\Oc_{X}(\Theta)\big)\iso \\
&\iso m_{\ast}(1\times \iota_{X})_{\ast}\big(\pi_{2}^{\ast}\Oc_{X}(\Theta)\otimes \pi_{1}^{\ast}\Oc_{X}(\Theta)\big)\iso \\
&\iso m_{\ast}(1\times \iota_{X})_{\ast}\big(\pi_{2}^{\ast}\Oc_{X}(\Theta)\otimes (1\times \iota_{X})^{\ast}\pi_{1}^{\ast}\Oc_{X}(\Theta)\big)\iso \\
&\iso m_{\ast}\big((1\times \iota_{X})_{\ast}\pi_{2}^{\ast}\Oc_{X}(\Theta)\otimes \pi_{1}^{\ast}\Oc_{X}(\Theta)\big)\iso \\
&\iso m_{\ast}\big(\pi_{2}^{\ast}\iota_{X}^{\ast}\Oc_{X}(\Theta)\otimes \pi_{1}^{\ast}\Oc_{X}(\Theta)\big)\iso \\
&\iso \Oc_{X}(\Theta)\star \Oc_{X}(\Theta)\,,
\end{align*}
where the last step uses the fact that $\Oc_{X}(\Theta)$ is symmetric; that is $\iota_{X}^{\ast}\Oc_{X}(\Theta)\iso \Oc_{X}(\Theta)$. Now, using Proposition \ref{eq:p:*} and bearing in mind that $\Sc \big(\Oc_{X}(\Theta)\big)\iso \Oc_{X}(-\Theta)$ one has: 
$$\Sc \big(\Oc_{X}(\Theta)\star \Oc_{X}(\Theta)\big)\iso \Sc\big(\Oc_{X}(\Theta)\big)\otimes \Sc \big(\Oc_{X}(\Theta)\big)\iso \Oc_{X}(-2\Theta)\,.$$
This implies that $\Sc(\Lca)\otimes \Lca \otimes \Oc_{X}(-\Theta)$ is $\IT_{0}$, and is therefore M-regular.
\end{proof}

\begin{corol}\label{eq:c:delta}
Let $H=\{c_{1},\dots ,c_{n+2}\}$ be a set of $n+2$ pairwise different points of $X$. The subscheme $Z^{i}(\Lca,H)$ of Definition \ref{eq:d:Z} coincides with the subscheme defined by the global equations $\Lambda^{i}\delta(\Lca,H)=0$, where:
$$\delta(\Lca,H)\colon \Gamma(X\times X,\pi_{1}^{\ast}\Lca \otimes \Pc \otimes \pi_{2}^{\ast}\Lca) \to \Gamma(X\times X,(\pi_{1}^{\ast}\Lca \otimes \Pc \otimes \pi_{2}^{\ast}\Lca)_{|H\times X})\,.$$

\end{corol}

\begin{proof} 
By definition, the subscheme $Z^{i}(\Lca,H)$ is defined by the local equations $\Lambda^{i}\alpha(\Lca,H)=0$, where $\alpha(\Lca,H)$ is the morphism of equation (\ref{eq:e:alpha}) (recall that now we are identifying $X$ with $\hat X$). This map induces a morphism:
\begin{equation}\label{eq:e:alphax1}
\alpha(\Lca,H)\otimes 1\colon \Sc(\Lca)\otimes \Lca \to \Sc(\Lca_{|H})\otimes \Lca\,.
\end{equation}
Let us denote by $\delta(\Lca,H)$ the morphism induced by $\alpha(\Lca,H)\otimes 1$ between the spaces of global sections:
$$\delta(\Lca,H)\colon \Gamma(X,\Sc(\Lca)\otimes \Lca) \to \Gamma(X,\Sc(\Lca_{|H})\otimes \Lca)\,,$$
or, what is the same:
$$\delta(\Lca,H)\colon \Gamma(X\times X,\pi_{1}^{\ast}\Lca \otimes \Pc \otimes \pi_{2}^{\ast}\Lca) \to \Gamma(X\times X,(\pi_{1}^{\ast}\Lca \otimes \Pc \otimes \pi_{2}^{\ast}\Lca)_{|H\times X})$$
by the properties of the pushforward. Since $\Sc(\Lca)\otimes \Lca$ is generated by its global sections, one concludes.
\end{proof}

Let $\{\theta_{\sigma}(z), \sigma \in (\Z/2\Z)^g\}$ be a basis for the vector space $\Gamma \big(X,\Oc_{X}(2\Theta)\big)$. When $k=\C$, the basis $\{\theta_{\sigma}\}$ could be the classical basis of second-order theta functions of $(X,\Theta)$. If the characteristic of $k$ is $\neq 2$, we fix a theta structure and the basis will be the Mumford theta functions (see \cite{MumEq1}).

\begin{thm}\label{eq:t:globaleq}
Let $\Lca$ be the line bundle $\tau_{\xi}^{\ast}\Oc_{X}(2\Theta)$ and $\xi \in X$ a point such that $2\xi=c_{1}+\cdots +c_{n+2}$. Then:
$$U^{n+2}(\Lca,H)=\varphi_{\Lca}^{-1}\big(Z^{n+2}(\Lca,H)\big)$$
is scheme-theoretically defined by the following system of global equations:
$$\det \big(\theta_{\sigma_{\lambda_{i}}}(z-\xi+c_{j})\big)=0$$
for every $(\sigma_{\lambda_{1}},\dots ,\sigma_{\lambda_{n+2}})\in \big((\Z/2\Z)^g\big)^{n+2}$. Furthermore, these equations are the pullback under the isogeny $\varphi_{\Lca}$ of the global equations of $Z^{i}(\Lca,H)$.
\end{thm}

\begin{proof}
By Definition \ref{eq:d:U}, the equations of $U^{n+2}(\Lca,H)$ are $\Lambda^{n+2}\beta(\Lca,H)$=0. In this context, $\beta(\Lca,H)$ is the map:
$$\beta(\Lca,H)\colon \Gamma(X,\Lca)\otimes_{k}\Oc_{X} \to \pi_{2 \ast}(\pi_{1}^{\ast}\Lca_{|(H,0)+\Delta})\iso \oplus_{i=1}^{n+2}\Gamma(X,\tau_{c_{i}}^{\ast}\Lca)\otimes_{k}\Oc_{X}\,.$$
Since $\Lca=\tau_{\xi}^{\ast}\Oc_{X}(2\Theta)$ for some $\xi \in X$, then 
$\{\theta_{\sigma}(z-\xi), \sigma \in (\Z/2\Z)^g\}$ is a basis for the vector space $\Gamma (X,\Lca)$ (recall that the translation map $\tau_{\xi}\colon X\to X$ is defined by $\tau_{\xi}(x)=x+\xi$). Therefore, taking global sections in the morphism above one has:
\begin{align*}
\beta(\Lca,H)\colon \Gamma(X,\Lca) &\to \oplus_{i=1}^{n+2}\Gamma(X,\tau_{c_{i}}^{\ast}\Lca)\\
\theta_{\sigma}(z-\xi) &\mapsto \big(\theta_{\sigma}(z-\xi +c_{1}),\dots ,\theta_{\sigma}(z-\xi+c_{n+2})\big)\,.
\end{align*}
 and thus the global equations of $U^{n+2}(\Lca,H)$ are:
$$\det \big(\theta_{\sigma_{\lambda_{i}}}(z-\xi+c_{j})\big)=0$$
for every $(\sigma_{\lambda_{1}},\dots ,\sigma_{\lambda_{n+2}})\in \big((\Z/2\Z)^g\big)^{n+2}$.

Let us see that these equations are the pullback under the isogeny $\varphi_{\Lca}$ of the global equations of $Z^{i}(\Lca,H)$. From Corollary \ref{eq:c:delta}, the global equations of $Z^{n+2}(\Lca,H)$ are $\Lambda^{i}\delta(\Lca,H)=0$ and now we have that:
{\small$$\Gamma(X\times X,\pi_{1}^{\ast}\Lca \otimes \Pc \otimes \pi_{2}^{\ast}\Lca) \xrightarrow{\delta(\Lca,H)} \Gamma \big(X\times X,(\pi_{1}^{\ast}\Lca \otimes \Pc \otimes \pi_{2}^{\ast}\Lca)_{H\times X}\big)\iso \oplus_{i=1}^{n+2}\Gamma(X,\tau_{\tilde c_{i}}^{\ast}\Lca)\,,$$
}where $\tilde c_{i}\in X$ is such that $2\tilde c_{i}=c_{i}$ (the last isomorphism makes use of the square lemma). Recall that $\delta(\Lca,H)$ is the map induced between the global sections of the map $\alpha(\Lca,H) \otimes 1$ of equation (\ref{eq:e:alphax1}).

Taking the pullback of $\alpha(\Lca,H) \otimes 1$ with respect to $\varphi_{\Lca}$ we obtain a morphism:
$$\Gamma(X,\Lca)\otimes (\Lca^{-1} \otimes 2_{X}^{\ast}\Lca) \to \oplus_{i=1}^{n+2}\big(\Gamma(X,\tau_{c_{i}}^{\ast}\Lca)\big)\otimes (\Lca^{-1} \otimes 2_{X}^{\ast}\Lca)\,.$$

Since $\Lca^{-1}\otimes 2_{X}^{\ast}\Lca$ is generated by its global sections (it is algebraically equivalent to $\Oc_{X}(m\Theta)$ for $m>2$), using Lemma \ref{eq:l:2} and Corollary \ref{eq:c:3} one concludes.
\end{proof}

\section{Jacobians.}\label{eq:s:Jac}\qquad

The aim of this section is to give a geometric meaning to the above equations in the case in which the principally polarized abelian variety (p.p.a.v.) is the Jacobian of a smooth projective curve $C$ of genus $g\geq 1$. We shall see, using Kempf's results, that $Z^{i}(\Lca,H)$ is a translation of the $i$-th symmetric product of the curve C. Moreover, when $H$ consists of $3$ pairwise different closed points in the Jacobian, then the equations computed in Theorem \ref{eq:t:globaleq} are Fay's trisecant identity (\cite{Fay}). If the degree of $H$ is $n+2$ the equations are Gunning's relations (\cite{Gun}).

\subsection{The Jacobian case and the relation with Kempf's results.}\qquad

Let $J$ be the Jacobian of a smooth projective curve $C$ of genus $g\geq 1$. Accordingly, $J$ is a p.p.a.v. and the principal polarization is the so called theta divisor $\Theta$, which is determined up to translation. 

We shall work with line bundles $\Lca$ that are algebraically equivalent to $\Oc_{J}(m\Theta)$ for $m>0$. 

Consider the following exact sequence of sheaves on $J$:
$$0\to \Ic_{C}\otimes \Lca \to \Lca \to \Lca_{|C} \to 0$$
where $\Ic_{C}$ denotes the sheaf of ideals of the curve $C$.
\begin{prop}
For $m>1$, the twisted ideal sheaf $\Ic_{C}\otimes \Lca$ satisfies the $\IT_{0}$ condition.
\end{prop}
\begin{proof}
The case $m=2$ is exactly \cite[Thm.4.1]{PP1}. Since $\Oc_{J}(\Theta)$ is also $\IT_{0}$ (its Fourier-Mukai transform is $\Oc_{J}(-\Theta)$), using \cite[Prop.2.9]{PP1} one has that $\Ic_{C}(3\Theta)$ is $\IT_{0}$. By applying this method recursively the result follows.
\end{proof}

\begin{corol}
The sheaf $\Lca_{|C}$ satisfies the $\IT_{0}$ condition.
\end{corol}

\begin{corol}
The map $\rho \colon \Sc(\Lca) \to \Sc(\Lca_{|C})$ is surjective.
\end{corol}

Let $H$ be a closed finite subscheme of $C$. Thus, we have a restriction map:
$$\alpha'(\Lca,H) \colon \Sc(\Lca_{|C})\to \Sc(\Lca_{|H})$$
and the map:
$$\alpha(\Lca,H) \colon \Sc(\Lca)\to \Sc(\Lca_{|H})$$
factorizes by $\rho$:
$$\xymatrix{
\Sc(\Lca) \ar[rr]^{\alpha(\Lca,H)} \ar@{->>}[dr]^{\rho} && \Sc(\Lca_{|H})\\
&\Sc(\Lca_{|C})\ar[ur]_{\alpha'(\Lca,H)} &
}$$

Therefore, the subscheme $Z^{i}(\Lca,H)$ (see Definition \ref{eq:d:Z}) coincides with the closed subscheme defined by $\Lambda^{i}\alpha'(\Lca,H)=0$. This scheme has been intensively studied in \cite{Ke} in dual form and its relevance lies in the following fact. If:
$$2g-2\geq mg-\deg H \geq g-1\,,$$
we define $i$ by: 
$$i=2g-2-mg+\deg H\,.$$
Therefore, $Z^{\deg H}(\Lca,H)$ is a translation of $-W^{i}$, where $W^{i}$ is the image under the Abel-Jacobi map of the $i$-th symmetric product of the curve $C$. 

Since $\varphi_{\Lca}$ consists of multiplying by $m$, one has that $m^{-1}Z^{\deg H}(\Lca,H)=U^{\deg H}(\Lca,H)$. Thus, in this case Theorem \ref{eq:t:globaleq} generalizes Fay's trisecant identity for $m=2$ and $\deg H=3$.

\subsection{Recovering Fay's Trisecant Identity and Gunning's relations.}\qquad 

The global equations obtained in Theorem \ref{eq:t:globaleq} generalize Fay's trisecant identity (\cite{Fay}) and Gunning's relations (\cite{Gun}). To see this, let $p_{0}\in C$ be a closed point and $i\colon C \hookrightarrow J$ the embedding defined by $p_{0}$. We fix a theta characteristic $\eta$ (that is, $\eta^{\otimes 2}\iso \omega_{C}$, where $\omega_{C}$ denotes the canonical line bundle on $C$). These data allow us to determine a unique polarization $\Theta$ on $J$ with the condition $\Theta_{|C}=\eta +p_{0}$ (where we also use $\eta$ to refer to its associated divisor).

Let $\{p_{1},\dots ,p_{n+2}\}$ be pairwise different points of $C$ and $\{c_{1},\dots ,c_{n+2}\}$ their images in $J$, and let us choose a point $\xi \in J$ such that $2\xi =c_{1}+\cdots +c_{n+2}$. Let us define $\Lca:=\tau_{\xi}^{\ast}\Oc_{J}(2\Theta)$ and $H=\{c_{1},\dots ,c_{n+2}\}\subset J$. 

By applying Theorem \ref{eq:t:globaleq} to  $(J,\Lca,H)$ we recover Fay's trisecant identity for $n=1$ and Gunning's relations for arbitrary $n$.

The geometric interpretation of the above results must be given in terms of the geometry of the Kummer variety associated with the Jacobian. Let $\{\theta_{\sigma}(z), \sigma \in (\Z/2\Z)^g\}$ be a basis for the vector space $\Gamma \big(X,\Oc_{X}(2\Theta)\big)$; the linear series $|2\Theta|$ is identified with the projective space $\Ps^{N}$ (where $N=2^g-1$), and one has a morphism:
\begin{align*}
\phi_{J}\colon J & \to |2\Theta|\iso \Ps^{N}\\
x & \mapsto \Theta_{x}+\Theta_{-x}=(\dots ,\theta_{\sigma}(x),\dots )
\end{align*}
whose schematic image is the Kummer variety $K(J)$ (where $\Theta_{x}$ is the image of $\Theta$ under the translation morphism $\tau_{x}$).

Let $p_{1},p_{2},p_{3}$ be three pairwise different points of $C$ and $c_{1},c_{2},c_{3}$ its images in $J$ and let $\xi$ be a point $J$ such that $2\xi=c_{1}+c_{2}+c_{3}$. Recall that we are identifying $J$ with $\hat J$, so the isogeny $\varphi_{\Oc_{X}(2\Theta)}$ consists of multiplying by $2$.

\begin{thm}
A point $x\in \varphi^{\ast}_{\Oc_{X}(2\Theta)}(W^{1}-2\xi)$, that is $2x+c_{1}+c_{2}+c_{3}\in W^{1}$, if and only if the points $\phi_{J}(x+c_{1}),\phi_{J}(x+c_{2}),\phi_{J}(x+c_{3})$ are collinear in $\Ps^{N}$.
\end{thm}
 
 This theorem and its schematic equivalent Theorem \ref{eq:t:globaleq} means that, up to the isogeny $\varphi_{\Oc_{X}(2\Theta)}$, the ideal of the curve $C$ as a subvariety of the Jacobian is generated by the trisecant identities. 

\subsection{Relation with the Riemann-Schottky problem.}\quad

A geometrical characterization of Jacobians was proposed by Gunning \cite{Gun1,Gun2} based on Fay's trisecant identity \cite{Fay}.  Gunning's result claims that the existence of a family of trisecants is a necessary and sufficient condition for a p.p.a.v. to be the Jacobian of an algebraic curve. The first step in formulating Gunning's criterion in terms of equations was taken by Welters \cite{Wel1,Wel2} and its remarkable conjecture, which states that to characterize Jacobians the existence of \emph{only one} trisecant, is sufficient. This conjecture has been solved by Krichever \cite{K2,K3} in all three different configurations of the intersections points of the trisecant in the Kummer variety. Here we offer here a reformulation of Krichever's result (\cite{K3}).

Let $(X,\Theta)$ be a p.p.a.v. and let $H=\{c_{1},c_{2},c_{3}\}$ be a set of $3$ pairwise different closed points of $X$. Let $\Lca$ be the line bundle $\Oc_{X}(2\Theta)$.

\begin{defin}
Let us define $W^{1}(X,H):=-Z^{3}(\Lca,H)$.
\end{defin}

Recall that $Z^{3}(\Lca,H)$ is defined by the vanishing of the map between exterior powers:
$$\Lambda^{3}\alpha(\Lca,H)\colon \Lambda^{3}\Sc(\Lca) \to \Lambda^{3}\Sc(\Lca_{|H})\,.$$

From the theory developed in the previous sections, one can reformulate Krichever's result (\cite{K3}) in the following fashion.

\begin{thm}
The p.p.a.v. $(X,\Theta)$ is the Jacobian of a curve if and only if there exists at least one point in $W^1(X,H)$ that is not a $2$-torsion point. 
\end{thm}


\bibliographystyle{siam}
\bibliography{biblio}
\end{document}